\documentclass[11pt,twoside,reqno]{amsart}
\usepackage{amsmath, amsfonts, amsthm, amssymb,amscd, graphicx, amscd}
\usepackage{float,epsf}
\usepackage[english]{babel}
\usepackage{enumerate}
\usepackage{tikz}
\usepackage{hyperref}
\usepackage{float} 
\usepackage{srcltx}
\usepackage{subfigure}
\usepackage{graphics}
\usepackage{epsfig}

\usepackage{geometry}
\usepackage{verbatim}
\usepackage{mathrsfs}

\newcommand{\R}{\mathbb{R}}

\newtheorem{theorem}{Theorem}[section]
\newtheorem{lemma}[theorem]{Lemma}

\newtheorem{proposition}[theorem]{Proposition}
\newtheorem{corollary}[theorem]{Corollary}
\newtheorem{remark}[theorem]{Remark}

\newtheorem{question}[theorem]{Question}

\newtheorem*{remark*}{Remark}

\numberwithin{equation}{section}
\numberwithin{figure}{section}

\makeatletter
\@namedef{subjclassname@2020}{Mathematics Subject Classification 2020}
\makeatother

\renewcommand{\div}{\mathrm{div}} 

\def\intave#1{\int_{#1}\hbox{\llap{$\raise2.3pt\hbox{\vrule
				height.9pt width7pt}\phantom{\scriptstyle{#1}}\mkern-2mu$}}}

\everymath{\displaystyle}

\usepackage[misc]{ifsym}

\author{Qinfeng Li}
\address{School of Mathematics, Hunan University, Changsha, Hunan, China.}
\email{liqinfeng1989@gmail.com}

\author{Hang Yang}
\address{School of Mathematics, Hunan University, Changsha, Hunan, China.}
\email{hangyang0925@gmail.com}

\thanks{Research of Qinfeng Li is supported by National Key R\&D Program of China (2022YFA1006900) and the National Science Fund of China General Program (No. 12471105).}


\begin{document}
\title[Shape Optimization Problems with Radially Decreasing Heat Source]{Heat Transfer Shape Optimization: Stability and Non-Optimality of the Ball}
	
	\begin{abstract}  
	This paper investigates shape optimization problems in the context of heat transfer, with a focus on the stability and non-optimality of round domains under Robin boundary conditions. Using the flow approach and Steklov eigenvalue estimates, we derive the necessary and sufficient stability conditions for a ball to maximize the averaged heat when the heat source is radially decreasing. Our results show that, counterintuitively, a ball may not be optimal for maximizing the averaged heat under heat convection, even with radially decreasing heat sources located on the center of the ball. Moreover, we identify stability-breaking phenomena by giving precise values of thresholds, which depend on the Robin coefficient, dimension, and volume constraints. Additionally, we demonstrate that a ball can maximize the averaged temperature under certain conditions and we also explore optimal shapes in thin insulation problems.
	\end{abstract}
	
	\maketitle
	
	\vskip0.2in

	\section{Introduction}
\subsection{Motivation}
Let $\Omega\subset\R^{n}$ be a bounded domain with Lipschitz boundary. We consider the following energy functional arising in convection heat transfer background, given by:
\begin{equation}
		J_{\beta}^{f}(\Omega):=\inf \Big\{\frac{1}{2}\int_{\Omega}|\nabla u|^{2} \, dx+\frac{\beta}{2}\int_{\partial\Omega}u^{2} \, d\sigma-\int_{\Omega}fu \, dx:u\in H^{1}(\Omega) \Big\},
		\label{eq:Robin energy}
\end{equation} 
where $f$ is a given non-negative function locally in $L^2(\mathbb{R}^n)$, $\sigma$ is the volume element of $\partial \Omega$ and $\beta>0$ is a constant related to the heat convection coefficient. The minimizer $u_{\Omega}^{f,\beta}$ of \eqref{eq:Robin energy} satisfies the Robin boundary value problem:
\begin{align}
		\begin{cases}
			-\Delta u=f \quad & \mbox{in $\Omega$}\\
			\frac{\partial u}{\partial\nu}+\beta u=0 \quad & \mbox{on $\partial \Omega$},
		\end{cases}
		\label{eq:Robin equations}
\end{align}
where $\nu$ denotes the outward unit normal to $\partial \Omega$. The solution $u_{\Omega}^{f,\beta}$ can be viwed as the steady-state temperature under the heat convection condition in the thermal body $\Omega$, and the function $f$ can be viewed as the heat source distributed in $\mathbb{R}^n$. The boundary condition of \eqref{eq:Robin equations} reflects Newton’s law of cooling combined with Fourier’s law of heat transfer, once we assume that the external temperature is zero.

When $\beta=+\infty$, the functional $J_\beta^f(\Omega)$ becomes the following:
\begin{align}
		J^{f}_{\infty}(\Omega):=\inf \Big\{\frac{1}{2}\int_{\Omega}|\nabla u|^2
  \,dx-\int_{\Omega} fu \,dx:u\in H^{1}_{0}(\Omega) \Big\}.
		\label{eq:Dirichlet energy}
\end{align} 
The minimizer $u_{\Omega}^{f}$ of \eqref{eq:Dirichlet energy} solves the Dirichlet problem: 
\begin{align}
	\label{Dirichletcondition}
	\begin{cases}
		-\Delta u=f \quad & \mbox{in $\Omega$,}\\
	 u=0 \quad & \mbox{on $\partial \Omega$.}
	\end{cases}
\end{align} 
The solution $u_\Omega^f$ then describes the steady-state temperature under the heat conduction condition in the thermal body $\Omega$, and again the external temperature is assumed to be $0$.

Applying the divergence theorem for  $J_{\infty}^{f}(\Omega)$ and $J_{\beta}^{f}(\Omega)$ yields the following identities:
\begin{align}
\label{average-heat}
-2J_{\infty}^{f}(\Omega)=\int_{\Omega}fu_{\Omega}^{f}dx\quad \mbox{and}\quad -2J_{\beta}^{f}(\Omega)=\int_{\Omega}fu_{\Omega}^{f,\beta}dx,
\end{align}
which often represent the total heat of the thermal body $\Omega$ under the conditions of heat conduction and heat convection, respectively. Therefore, the problem of finding the optimal shape to maximize the averaged heat under these conditions is equivalent to the problem of minimizing the energy functionals $J^{f}_{\infty}(\cdot)$ and $J^{f}_{\beta}(\cdot)$ among domains with fixed volume. We also mention that similar energy minimization problems from heat transfer background are considered in \cite{BW}, \cite{BBN}, \cite{BBN1}, \cite{Buttazzo}, \cite{Dam}, \cite{DP00}, \cite{DLW}, \cite{HLL1}, etc.

 
$f\equiv 1$ is a classical special case for such problems. In this case, the quantity $\int_{\Omega}u_{\Omega}^{1}dx$ is the torsional rigidity of $\Omega$. The classical Saint-Venant inequality states that, among domains of fixed volume, the ball maximizes the torsional rigidity, and therefore 
\begin{align}
\label{eq:SV}
J^{1}_{\infty}(\Omega^{\sharp})\leq J^{1}_{\infty}(\Omega),
\end{align}
where $\Omega^{\sharp}$ denotes the ball centered at the origin with the same volume as that of $\Omega$.
More recently, Bucur-Giacomini \cite{Bucur} and Alvino-Nitsch-Trombetti \cite{AN} established this same optimality property of balls in Robin problems. They proved that for all constant $\beta>0$, the balls always minimize $J_{\beta}^{1}(\cdot)$ among domains with fixed volume. That is,
\begin{align}
\label{eq:BG}
J^{1}_{\beta}(\Omega^{\sharp})\leq J^{1}_{\beta}(\Omega).
\end{align}

Inequalities \eqref{eq:SV} and \eqref{eq:BG} demonstrate that in both heat conduction and heat convection models, with a uniformly distributed heat source (i.e., $f$ is a positive constant) and a fixed volume of the thermal body, the total heat is maximized when the shape is round. This raises the question of whether similar isoperimetric results apply when the heat source $f$ decreases radially.

This question is rather natural from the physical point of view, since the radially decreasing $f$ models heat sources centered at a single point with intensity diminishing outward——a common scenario in many thermal design problems. For such a heat source, a crucial question is  to determine the optimal shape and the position of the center of the source to maximize the total heat inside the thermal body. When the volume of the thermal body is fixed, it is reasonable to expect that for a point heat source located at the origin, the total heat \eqref{average-heat} is maximized when the domain is as close as possible to the heat source. The most natural choice, which is also from the mathematical taste in view of \eqref{eq:SV} and \eqref{eq:BG}, is to choose the thermal body to be a ball and place the center of the radially decreasing heat source at the center of the ball.

For Dirichlet conditions, this intuition is indeed correct, which can be easily obtained as applications of the maximum principle and the classical Talenti's pointwise comparison result \cite{Talenti} for Dirichlet boundary conditions, see Proposition \ref{thm:global} below in this paper. 

However, an interesting and somewhat counterintuitive phenomenon arises in the Robin boundary conditions: even when the domain is a ball centered at the maximal point of the radially decreasing function $f$, we can show that it may even fail to be a local maximizer of the total heat (i.e. the local minimizer of $J_{\beta}^{f}(\Omega)$).

To present the result rigorously, we first introduce some terminologies. We let $F(t,x)$ be the \textit{flow map} generated by a smooth vector field $\eta \in C^{\infty}(\mathbb{R}^n, \mathbb{R}^n)$, that is, 
	\begin{align*}
		\begin{cases}
			\frac{d}{dt}F(t,x)=\eta\circ F(t,x)\quad  &t \ne 0\\
			F(0,x)=x\quad & t=0.
		\end{cases}
	\end{align*}
	We also denote $F(t,x)$ by $F_t(x)$, and hence $F_t$ is a local diffeomorphism when $|t|$ is small. We say that $F_t$ or $\eta$ preserves the volume of $\Omega$, if $|F_t(\Omega)|=|\Omega|$. Let $\mathcal{E}(\cdot)$ be a $C^2$ shape functional, and we say that $\Omega$ is \textit{stationary} to $\mathcal{E}(\cdot)$, if $\frac{d}{dt}\Big|_{t=0} \mathcal{E}(F_t(\Omega))=0$ for any smooth flow map $F_{t}(\cdot)$ preserving the volume of $\Omega$ along the flow. We say that $\Omega$ is \textit{stable} to $\mathcal{E}(\cdot)$, if $\Omega$ is stationary to $\mathcal{E}(\cdot)$ and $\Omega$ satisfies $\frac{d^2}{dt^2}\Big|_{t=0} \mathcal{E}(F_t(\Omega))\ge 0$ for any smooth volume preserving flow map $F_t(\cdot)$. 

With the notations above, we are able to establish the necessary and sufficient condition for the stability of ball centered at the origin, when the heat source $f$ is radial.
\begin{theorem}
\label{thm:Robin stable iff condition}
Let $f>0$ be a smooth radial function about the origin, and $B_{R} \subset \mathbb{R}^n$ be the ball of radius of $R$ centered at the origin. Then for any $\beta>0$, $B_{R}$ is a stationary shape with respect to $J_{\beta}^{f}(\cdot)$ defined in (\ref{eq:Robin energy}). Moreover, $B_{R}$ is a stable shape with respect to $J_{\beta}^{f}(\cdot)$ if and only if
\begin{equation}
\left(f(R)-\frac{n-1-R\beta}{n}\bar{f}(R) \right)\left(f(R)-\bar{f}(R)\right)+\frac{1+\beta R}{n\beta}f_{r}(R)\bar{f}(R)\leq 0,
\label{eq: Robin stable iff condition}
\end{equation}
where $\bar{f}(R)=\frac{1}{|B_{R}|}\int_{B_{R}}f(x) \, dx$, and $f(R)$ and $f_{r}(R)$ are the values of $f(x)$ and the directional derivative of $f(x)$ in the radial direction restricted to $\partial B_{R}$, respectively.
\end{theorem}

In fact, our proof indicates that if \eqref{eq: Robin stable iff condition} is violated, then $J_\beta^f(\cdot)$ strictly increases after a translation of $B_R$. The translation variation also plays an important role in our analysis, which essentially motivates us to use the Steklov eigenvalue estimates to derive the necessary and sufficient condition \eqref{eq: Robin stable iff condition}.

To further clarify the stability condition of $B_{R}$ for a given radially decreasing function $f$, we rewrite the inequality \eqref{eq: Robin stable iff condition} as follows:
\begin{align}
\label{eq:rearrange}
A_{0}\leq \beta A_{1}+\frac{1}{\beta}A_{2},
\end{align}
where $A_{0}=\Big(f(R)-\bar{f}(R)\Big)\Big(f(R)-\frac{n-1}{n}\bar{f}(R)\Big)+\frac{R}{n}f_{r}(R)\bar{f}(R)$, $A_{1}=\frac{ R}{n}\bar{f}(R)(\bar{f}(R)-f(R))$ and $A_{2}=-\frac{1}{n}f_{r}(R)\bar{f}(R)$. When $f$ is radially decreasing, $A_1, A_2\ge 0$, and $A_1=0$ if and only if $f$ is a constant in $B_R$. Also, $A_2=0$ if and only if $f_r(R)=0$. The sign of $A_0$ is not determined. 

Clearly, when $A_0\le 2\sqrt{A_1A_2}$, the ball $B_{R}$ is always stable for any $\beta>0$. Note that roughly speaking, the condition $A_0 \le 2\sqrt{A_1A_2}$ is automatically satisfied if the function $f$ does not decrease too fast in the radial direction. For instance, if $f$ satisfies $f(R) \geq \frac{n-1}{n} \bar{f}(R)$, then $A_0\le 0$, and thus $B_R$ is always stable for $J_{\beta}^{f}(\cdot)$ for all $\beta \in (0, \infty)$. A typical example is the constant source $f \equiv 1$, for which the stability condition \eqref{eq: Robin stable iff condition} holds for every $\beta > 0$, which matches known results established in \cite{Bucur} and \cite{AN}.


However, it is also possible for $A_0>2\sqrt{A_1A_2}$ to occur. For example, consider the family of radially decreasing functions $f_{\delta}(x)=\tfrac{1}{\delta^n}e^{-\pi\tfrac{|x|^2}{\delta^2}}$. As $\delta\to 0$, $\bar{f_{\delta}}(R)$ tends to $1/|B_{R}|$, but both $f_{\delta}(R)$ and $(f_{\delta})_{r}(R)$ tend to 0. Hence $A_0$ tends to $\tfrac{n-1}{n}\tfrac{1}{|B_R|^2}$, while $A_1A_2$ tends to $0$. Consequently, for such $f_\delta$ with sufficiently small $\delta>0$, $A_0>2\sqrt{A_{1}A_{2}}>0$. Then, according to \eqref{eq:rearrange}, there exist two positive constants $0<\beta_1<\beta_2$, where $\beta_1, \beta_2$ only depend on $f_{\delta}$, $n$ and $R$, such that $B_{R}$ is stable for $J_{\beta}^{f}(\cdot)$ when $\beta\in (0,\beta_{1}]\cup [\beta_{2},\infty)$, while $B_{R}$ is unstable when $\beta\in (\beta_{1},\beta_{2})$.

This phenomenon is rather counter-intuitive not only from the physical point of view, but also somehow violates the mathematical expectation. When $\beta$ is large, it is expected that the problem is more close to the Dirichlet case, and according to Proposition \ref{thm:global}, one might anticipate that if $\beta$ is larger than a threshold, then $B_R$ is stable. See also the Remark \ref{rmk-iff}. However, when the inequality
$A_0 > 2\sqrt{A_1 A_2}>0$
holds, the equivalent stability condition
$A_0 \leq \beta A_1 + \frac{1}{\beta}A_2$
fails precisely in an intermediate range of $\beta$, while being satisfied in both cases when the parameter is small or large. 

Last, we note that if $ \beta R\ge (n-1)$, then \eqref{eq: Robin stable iff condition} is always satisfied for any radially decreasing $f$. However, when $0<\beta R<n-1$, we can find a radially decreasing function $f$ such that \eqref{eq: Robin stable iff condition} is not satisfied, and thus $B_R$ is not stable to $J^f_\beta(\cdot)$. Such a function $f$ can be chosen to be like $\tfrac{1}{\delta^n}e^{-\pi\tfrac{|x|^2}{\delta^2}}$ for sufficiently small $\delta>0$ depending on $\beta$ and $R$, and the reason is similar to the previous discussion for the possibility of $A_0>2\sqrt{A_1A_2}$.

We put some of the above comments into the following corollary, summarizing some interesting phenomena:
\begin{corollary}
\label{cor1}
Let $f>0$ be a smooth radially decreasing function about the origin, and we use the same notations as before. Then, we obtain the following consequences.
\begin{itemize}
\item[(1)] When $f(R)\ge \frac{n-1}{n}\bar{f}(R)$, $B_{R}$ is stable to $J_{\beta}^{f}(\cdot)$ for all $\beta>0$.
\item[(2)] For any given $R>0$, $f$ satisfies either $A_0\le 2\sqrt{A_1A_2}$ or $A_0>2\sqrt{A_1A_2}$, where $A_0, A_1$ and $A_2$ are defined right after \eqref{eq:rearrange}. In the first scenario, $B_R$ is stable to $J^f_\beta(\cdot)$ for all $\beta>0$. In the second scenario, if additionally $f_r(R)<0$, that is $A_2>0$, then there exist two constants $0<\beta_{1}<\beta_{2}$ depending on $f$, $n$ and $R$, such that $B_{R}$ is stable to $J_{\beta}^{f}(\cdot)$ when $\beta\in(0,\beta_{1}]\cup [\beta_{2},+\infty)$, but is unstable to $J^f_\beta(\cdot)$ when $\beta\in(\beta_{1},\beta_{2})$.  
\item[(3)] Fix any $R>0$. If $\beta\ge \tfrac{n-1}{R}$, then for any smooth radially decreasing $f$, $B_R$ is always stable to $J^f_\beta(\cdot)$. However, if $0<\beta<\tfrac{n-1}{R}$, then there always exists a smooth radially decreasing function $f$ such that $B_R$ is not a stable shape to $J^f_\beta(\cdot)$ for this given $\beta$.
\item[(4)] Fix any $\beta>0$. If $R\ge \tfrac{n-1}{\beta}$, then for any smooth radially decreasing $f$, $B_R$ is always stable to $J^f_\beta(\cdot)$. However, if $0<R< \tfrac{n-1}{\beta}$, then there always exists a smooth radially decreasing function $f$ such that $B_R$ is not a stable shape to $J^f_\beta(\cdot)$ for this given $R$.
\end{itemize}
\end{corollary}

The statements (2)-(4) in Corollary \ref{cor1} can be viewed as stability breaking results, which depend not only on the shape of the radially decreasing heat source, but also on the volume constraints, the Robin coefficient and the dimension. 

In summary, for the Robin coefficient $\beta \in (0,+\infty)$, we have demonstrated that the inequality $\int_{\Omega} fu_\Omega^{f,\beta}\le \int_{\Omega^\sharp}fu_{\Omega^\sharp}^{f,\beta} $, or equivalently $J_\beta^f(\Omega^\sharp)\le J_\beta^f(\Omega)$, is in general not true for arbitrary radially decreasing heat source $f$. Nevertheless, applying Talenti's comparison result, when $\beta=+\infty$, this inequality is indeed valid for any radially decreasing positive function $f$. Also, as an application of \cite{AN}, we prove some modified inequalities regarding the averaged temperature, which hold for any $\beta \in (0,+\infty)$ and arbitrary radially decreasing positive function $f$, and we also obtain some optimization results on a thin insulation problem for radial heat sources, which appropriately extends a result proved in \cite{DNST} in two dimensions. These global results are proved in section 4.

Finally, as a byproduct of Theorem \ref{thm:Robin stable iff condition}, we show that Talenti type pointwise comparison for Robin problems is not true in general even if $f>0$ is radially decreasing and the domain under consideration is a small translation of a ball centered at the maximal point of such $f$. This byproduct via stability argument further indicates that the validity of Robin pointwise comparison is not only sensitive to dimension as pointed out in \cite{AN}, but may also depend on the size of $\beta$ and the volume of the domain. For details, see section 4.

\vskip 0.3cm
\textbf{The techniques.} To study local optimality of domains, it is natural to use the shape derivative method. The shape derivative method has a long history and goes back to Hadamard \cite{H}, and has been studied by many authors thereafter. One can find in the comprehensive monographs \cite{Henrot2} and \cite{Henrot} many related references, as well as shape derivative formulas for some classical shape functionals such as the first eigenvalue of Dirichlet Laplacian, the second eigenvalue of Neumann Laplacian and so on, which are all derived by considering the deformation map $F_t(x)=x+t\eta(x)$ for sufficiently regular vector field $\eta$.
	
Inspired by \cite{DLW} and \cite{HLL1}, we instead use smooth geometric flows which preserve the volume of the domain, to derive the stability condition of the shape functional $J^f_\beta(\cdot)$. The use of the flow map does not lose generality if one considers smooth variations, and the advantage is that the evolution equation formulas stated in the flow language can be handled in a neater way via formulas from differential geometry, especially when dealing with boundary integrals. Recently, the flow methods are also applied to obtain new global monotonicity properties on shape functionals, see \cite{He2025} and \cite{H2025}.  

Here we briefly explain how we establish \eqref{eq: Robin stable iff condition} as a necessary and sufficient condition for the stability of $B_R$ to energy functional $J_\beta^f(\cdot)$. First, to prove that the condition \eqref{eq: Robin stable iff condition} is sufficient for the stability of $B_R$, a crucial step is to prove a  sharp inequality relating the Robin state function, its shape derivative and the velocity of the flow map. This can be achieved by using estimates of the second Steklov eigenvalue of the Laplacian. Such argument is indeed natural, since the shape derivative of the state function is a valid trial function in the variational characterization of the second Steklov eigenvalue. Combining some delicate analysis, we are able to derive the sufficiency of \eqref{eq: Robin stable iff condition}. Second, we observe that when the flow is chosen to be a translation, the equality case happens in the Steklov eigenvalue estimate, and this also implies the necessity of the condition \eqref{eq: Robin stable iff condition}.

\vskip 0.3cm

\textbf{Outline of the paper.} In section 2, we state some formulas in differential geometry used in later computation of shape derivatives via the flow method. In section 3, we derive the first and second shape derivatives of $J_{\beta}^{f}(\cdot)$ along volume-preserving flows, and we also prove Theorem \ref{thm:Robin stable iff condition}. In section 4, we prove some global optimization results with respect to both Dirichlet and Robin problems, and we also discuss the counterexamples on Robin Talenti comparison results. In section 5, we prove the formulas stated in section 2 as an appendix.

	\section{Preliminaries}
	In order to implement second variation on our shape functionals, we first review some basic formulas in differential geometry.
	
	Let $\eta$ be a smooth vector field defined in $\mathbb{R}^n$, and let $M$ be a smooth orientable hypersurface in $\mathbb{R}^n$. Then the tangential gradient of $\eta$ along $M$, denoted by $\nabla ^M \eta$, is defined as the linear map or $n \times n$ matrix defined on $M$, such that for any point $p \in M$ and any vector $V$, the following identity holds:
	\begin{align*}
		\nabla^M \eta V=\nabla \eta V^T,
	\end{align*}where $V^T$ is the projection of $V$ on $T_pM$, the tangent space of $M$ at $p$. Hence
	\begin{align*}
		\nabla^M \eta =\nabla \eta (I-\nu \otimes \nu),
	\end{align*}where $I$ is the identity map, $\nu$ is the unit normal to $T_pM$ with respect to the orientation of $M$, and $a\otimes b$ is the linear map defined as $(a\otimes b) V=(b\cdot V)a$. 
	
	If $(x^1,\cdots, x^{n-1})$ is a local coordinate system of $M$, then the tangential gradient of $\eta$ along $M$ can also be written in the following form
	\begin{align}
		\label{tangentialgradient}
		\nabla^M \eta =g^{\alpha\beta}\partial_\alpha \eta \otimes \partial_\beta F,
	\end{align}where $F$ is the position vector for points on $M$, $\partial_{\alpha}F=\frac{\partial F}{\partial x^\alpha}$, $g_{\alpha \beta}=<\partial_\alpha F, \partial_\beta F>$, $g^{\alpha\beta}$ is the inverse of the metric tensor $g_{\alpha \beta}$ and $\partial_\alpha \eta=(\nabla \eta)\partial_\alpha F$. In the above, the notation of repeated indices means summations, and $\alpha, \beta$ are labeled from $1$ to $n-1$.
	
	The tangential divergence of $\eta$ on $M$ is defined as the trace of $\nabla^M \eta$, that is,
	\begin{align}
		\label{tangentialdivergence}
		\div_{M}\eta=\div \eta -<\nabla \eta \nu, \nu> =g^{\alpha \beta}\partial_\alpha\eta \cdot \partial_\beta F,
	\end{align}where both $<\cdot,\cdot>$ and $\cdot$ denote the inner product in $\mathbb{R}^n$. 
	
	We also adopt the convention that given a function $f$ defined in $\mathbb{R}^n$, $\nabla_\alpha\nabla_\beta f$ denotes $\nabla_M^2 f(\partial_\alpha F,\partial_\beta F)$, where $\nabla_M^2 f$ is the Hessian of $f$ on $M$, and $\partial_\alpha\partial_\beta f$ denotes usual derivatives of $f$ first along $\partial_\beta F$ and then along $\partial_\alpha F$. One can check that the following identity holds
	\begin{align*}
		<\nabla^M(\nabla^M f)\partial_\alpha F, \partial_\beta F>=\nabla_M^2 f(\partial_\alpha F, \partial_\beta F)=\partial_\alpha \partial_\beta f-\Gamma_{\alpha\beta}^\gamma\partial_\gamma f,
	\end{align*}where $\Gamma_{\alpha \beta}^\gamma$ denotes Christoffel symbols and $\nabla^M f=\nabla f-\frac{\partial f}{\partial \nu}\nu$ on $M$, which can be extended in a neighborhood of $M$ by regarding $\nu$ as the gradient of the signed distance function. Hence we can regard the Hessian $\nabla_M^2 f$ as an $n\times n$ matrix given by $\nabla^M (\nabla^M f)$. Note that 
	\begin{align}
		\label{Hessian}
		\nabla^2 f=\nabla_M^2 f+\nabla^2 f \nu \otimes \nu+\nu \otimes \nabla^M \frac{\partial f}{\partial \nu}+\frac{\partial f}{\partial \nu}\nabla^M \nu,
	\end{align}
	where $\nabla ^M \nu$ is understood as $\nabla^M (\nabla d)$, where $d$ is the standard signed distance function, and thus $\frac{\partial f}{\partial \nu}$ is defined in a neighborhood of $M$ as $\nabla f\cdot \nabla d$. Hence we have the following two equivalent definitions of Laplacian-Beltrami operator of $f$ on $M$, namely
	\begin{align*}
		\Delta_M f:=tr(\nabla_M^2 f)= g^{\alpha \beta}\nabla_\alpha\nabla_\beta f=\Delta f-\frac{\partial f}{\partial \nu}H-f_{\nu\nu},
	\end{align*}where $H$ is the mean curvature of $M$ and $f_{\nu\nu}=tr(\nabla^2 f \nu \otimes \nu)$.
		
	The following proposition is proved in \cite[Section 2]{HLL1}, which will be used later in the computations of the first and second variations of $J_{\beta}^{f}(\cdot)$.
	\begin{proposition}
		\label{geometricevolution}
		Let $F_t(x):=F(t,x)$ be the flow map generated by a smooth vector field $\eta$, $\eta(t)=\eta\circ F_t$, $M_t=F_t(M)$, and $\sigma_t$ be the volume element of $M_t$. For any $p\in M$, we let $\nu(t)(p)$ be the unit normal to $M_t$ at $F_t(p)$, and $h(t)(p)$ be the second fundamental form on $M_t$ at $F_t(p)$.  Then, we have 
		\begin{align}
			\label{evolutionofarea2}
			\frac{d}{dt}d\sigma_t=(\div_{M_t} \eta)d\sigma_t,
		\end{align}
		\begin{align}
			\label{evolutionofnormalspeed}
			\frac{d}{dt}\left(\eta(t)\cdot \nu(t)\right)=(\eta(t)\cdot\nu(t)) (\div \eta -\div_{M_t}\eta)\circ F_t,
		\end{align}
		and
		\begin{align}
			\label{hijevolution}
			h_{\alpha\beta}'(t)(p)=-<\nabla_\alpha\nabla_\beta\eta(F_t(p)),\nu(t)(p)>, 
		\end{align}where $h_{\alpha\beta}(t)(p)$ and $\nabla_\alpha\nabla_\beta\eta(F_t(p))$ are the $\alpha,\beta$-components of $h(t)$ and the Hessian of $\eta$ on $M_t$ at $F_t(p)$, respectively under the original local coordinates.
		
		If $M$ is an $(n-1)$-sphere of radius $R$, then we also have
		\begin{align}
			\label{evoonsphere}
			\frac{d}{dt}\Big|_{t=0}H(t)=-\Delta_M(\eta\cdot \nu)-\frac{n-1}{R^2}\eta\cdot\nu,  
		\end{align}
        where $H(t)$ is the mean curvature on $M_t$.
	\end{proposition}
	For the convenience of readers, we will prove the above proposition in the appendix. The proof contains more details than that in \cite{HLL1}.
	
	The next proposition can be either deduced from the evolution equation of $det(\nabla F_t(x))$, or from the following well-known formula
	\begin{align}
    \label{wkformula}
		\frac{d}{dt}\int_{F_t(\Omega)}f(t,x)\, dx=\int_{F_t(\Omega)}f_t(t,x)\, dx+\int_{\partial F_t(\Omega)}f(t,x)\eta \cdot \nu\, d\sigma_t,
	\end{align}
where $\nu$ is the unit outer normal to the boundary. 

\begin{proposition}
 \label{vp}
		If $F(t,\cdot)$ is the flow map generated by a smooth vector field $\eta$ which preserves the volume of $\Omega$ along the flow, then for any $t$, we have
		\begin{align*}
			\int_{F_t(\Omega)}\div \eta\, dx=0\quad \mbox{and} \quad  \int_{F_t(\Omega)}\div\left( (\div\eta)\eta\right)\, dx =0.
		\end{align*}
	\end{proposition}

	
	\section{Proof of Theorem \ref{thm:Robin stable iff condition}} 
	We first stipulate some notations. Let $F(t,\cdot)$ be the flow map generated by a smooth vector field $\eta$ preserving volume, and let $u(t)$ be the solution to \eqref{eq:Robin equations} over the domain $F_t(\Omega)$. That is, $u(t)$ is the unique function over $F_t(\Omega)$ such that 
	$$J_{\beta}^{f}(F_t(\Omega))=\frac{1}{2}\int_{F_t(\Omega)}|\nabla u(t)|^{2} \, dx+\frac{\beta}{2}\int_{\partial F_t(\Omega)}(u(t))^{2} \, d\sigma_{t}-\int_{F_t(\Omega)}fu(t) \, dx.$$
	For small $|t|$, we also define $u'(t)$ as follows.
	\begin{align}
		u'(t)(F_t(x))=\frac{d}{dt} \left(u(t)(F_t(x))\right)-\nabla u(F_t(x))\cdot \eta (F_t(x)),
		\label{eq:derivative of solution u}
	\end{align}
	which is well-defined by standard regularity theory and implicit function theorem, see \cite{Henrot}.
	\subsection{First Variation and Stationarity Condition}
	\begin{proposition}
		Let $\Omega$ be a bounded smooth domain in $\mathbb{R}^n$, $f>0$ be a smooth function and
		$F_{t}(\cdot)$ be the map generated by a smooth vector field $\eta$ preserving the volume of $\Omega$ along the flow. Then for $|t|$ small, we have
		\begin{equation}
			\frac{d}{dt}J_{\beta}^{f}(F_{t}(\Omega))=\int_{\partial F_{t}(\Omega)}\big(-\beta^{2}u^{2}(t)+\frac{1}{2}|\nabla u(t)|^{2}+\frac{\beta}{2}u^{2}(t)H(t)-fu(t)\big)\zeta \, d\sigma_{t},
			\label{eq:Robin energy-first-derivative}
		\end{equation}
		where $H(t)$ is the mean curvature of $\partial F_{t}(\Omega)$, $\zeta=\eta\cdot\nu$
		and $\sigma_{t}$ is the volume element on the boundary $\partial F_{t}(\Omega)$. 
	\end{proposition}
	
	\begin{proof}
		
		By direct calculation, we have
		\begin{align*}
			I_{1}&=\frac{d}{dt}\left(\frac{1}{2}\int_{F_{t}(\Omega)}|\nabla u(t)|^{2} \, dx\right)\\
			&=\int_{F_{t}(\Omega)}\nabla u(t)\nabla u^{\prime}(t) \, dx+\int_{F_{t}(\Omega)}\frac{1}{2}|\nabla u(t)|^{2}\zeta \, d\sigma_{t}\\
			&=\int_{F_{t}(\Omega)}fu^{\prime}(t) \, dx+\int_{\partial F_{t}(\Omega)}\frac{\partial u(t)}{\partial\nu}u^{\prime}(t) \, d\sigma_{t}+\int_{\partial F_{t}(\Omega)}\frac{1}{2}|\nabla u(t)|^{2}\zeta \, d\sigma_{t}.
		\end{align*}
		\begin{align*}
			I_{2}&=\frac{d}{dt}\left(\frac{\beta}{2}\int_{\partial F_{t}(\Omega)}u^{2}(t) \, d\sigma_{t}\right)\\
			&=\frac{\beta}{2}\int_{\partial F_{t}(\Omega)}\left(2u(t)(u^{\prime}(t)+\nabla u(t)\cdot\eta)+u^{2}(t) \div_{\partial F_{t}(\Omega)}\eta\right) \, d\sigma_{t}\\
			&=\frac{\beta}{2}\int_{\partial F_{t}(\Omega)}\left(2u(t)u^{\prime}(t)+2u(t)\frac{\partial u(t)}{\partial\nu}\zeta+u^{2}(t)H(t)\zeta\right) \, d\sigma_{t}.
		\end{align*}
		\begin{equation*}
			I_{3}=\frac{d}{dt}\left(-\int_{F_{t}(\Omega)}fu(t) \, dx\right)=-\int_{F_{t}(\Omega)}fu^{\prime}(t) \, dx-\int_{\partial F_{t}(\Omega)}fu(t)\zeta \, d\sigma_{t}.
		\end{equation*}
		In the above, we have used the equation of $u(t)$, \eqref{evolutionofarea2} and \eqref{wkformula}.
		Summarizing the equations above, we have 
		\begin{align*}
			\frac{d}{dt}J_{\beta}^{f}(F_{t}(\Omega))&=I_{1}+I_{2}+I_{3}\\
			&=\int_{\partial F_{t}(\Omega)}\left(-\beta^{2}u^{2}(t)+\frac{1}{2}|\nabla u(t)|^{2}+\frac{\beta}{2}u^{2}(t)H(t)-fu(t)\right)\zeta \, d\sigma_{t}.
		\end{align*}
		
	\end{proof}
	
	In view of Proposition \ref{vp}, $\frac{d}{dt} \Big|_{t=0} J_{\beta}^{f}(F_{t}(\Omega))=0$ if and only if $-\beta^{2}u^{2}+\frac{1}{2}|\nabla u|^{2}+\frac{\beta}{2}u^{2}H-fu $
	on the boundary $\partial\Omega$ is a constant. We write this fact as the following corollary.
	
	\begin{corollary}
		\label{cor:Robin-shape-stationary}
		Let $\Omega$ be a bounded smooth domain in $\mathbb{R}^n$ and $f>0$ be a smooth function. Then, $\Omega$ is a stationary shape to $J_{\beta}^{f}(\cdot)$ if and only if there
		exists a solution to the following system of equations
		\begin{align}
			\label{eqs:Robin-stationary}
			\begin{cases}
				-\Delta u =f & \textit{in $\Omega$}\\
				\frac{\partial u}{\partial \nu}+\beta u=0 & \textit{on $\partial \Omega$} \\
				-\beta^{2}u^{2}+\frac{1}{2}|\nabla u|^{2}+\frac{\beta}{2}u^{2}H-fu=\textit{constant} & \textit{on $\partial \Omega$},
			\end{cases}
		\end{align}
		where $H$ is the mean curvature of $\partial\Omega$. 
	\end{corollary}

\begin{remark}
    When $f \equiv 1$ and $\beta =\infty$, ball is the unique shape in order for the over-determined system \eqref{eqs:Robin-stationary} to admit a solution. This is due to Serrin's famous result \cite{Serrin}. For $\beta \in (0,\infty)$, it is interesting to study the rigidity of the system, and whether or not the ball is the unique solution to \eqref{eqs:Robin-stationary} remains open to us, even in the case when $f\equiv 1$.
\end{remark} 
	
As a consequence of Corollary \ref{cor:Robin-shape-stationary}, when $f$ is radial about the origin, any ball centered at the origin is stationary to $J_{\beta}^{f}(\cdot)$. Now assuming that $f$ is radial, we proceed the computation of the second variation. From now on, we write $f(x)=f(r)$, where $r=|x|$.
	
	\subsection{Second Variation and Stability Condition}
	
	\begin{proposition}
		\label{pro:Robin-second variation}
		Let $B_{R}\subset\mathbb{R}^{n}$ be a ball of radius of $R$ centered at the origin and $f(x)> 0$ be
		a nonnegative smooth radial function. Let $F_{t}$ be the flow map generated
		by a smooth vector field $\eta$ preserving the volume of $B_R$, and we denote $u^{\prime}(0)$ by $v$. Then we have
		\begin{equation}
			\begin{aligned}
				\label{eq:Robin energy-second-derivative}
				\frac{d^{2}}{dt^{2}}\Big|_{t=0}J_{\beta}^{f}(F_{t}(B_{R}))=&\frac{\beta}{2}\left(u(R)\right)^{2}\int_{\partial B_{R}}\left(-\Delta_{\partial B_R}\zeta-\frac{(n-1)}{R^{2}}\zeta\right)\zeta  \, d\sigma-\int_{\partial B_{R}}f_{r}u\zeta^{2} \, d\sigma\\
				&+\int_{\partial B_{R}}(v\zeta+u_{r}\zeta^{2})(u_{rr}-\beta^{2}u) \, d\sigma,
			\end{aligned}
		\end{equation}
		where $\zeta=\eta\cdot\nu$, and $\nu$ is the outer unit normal to $\partial B_{R}$.
	\end{proposition}
	
	\begin{proof}
In the following, we write $u=u(0)$ for simplification.	We first derive the equations for  
		$v$. Let $\phi$ be a test function, and since $\frac{\partial u}{\partial\nu}+\beta u=0$ on $\partial B_{R}$ and $u$ is radial, we have
		\begin{align*}
			0=&\frac{d}{dt} \Big|_{t=0}\int_{\partial F_{t}(B_{R})}\Big(\frac{\partial u(t)}{\partial\nu}+\beta u(t)\Big)\phi \, d\sigma_{t}\\
			=&\int_{\partial B_{R}} \Big(\frac{\partial v}{\partial \nu}+u_{rr}\zeta+\beta v+\beta \nabla u \cdot \eta\Big) \phi \, d\sigma+\int_{\partial B_{R}}\Big(\frac{\partial u}{\partial \nu}+\beta u\Big)\nabla \phi \cdot \eta \,  d\sigma\\
			&+\int_{\partial B_{R}} \Big(\frac{\partial u}{\partial \nu} +\beta u\Big)\phi \div_{\partial B_{R}} \eta \, d\sigma\\
			=&\int_{\partial B_{R}} \Big(\frac{\partial v}{\partial \nu}+u_{rr}\zeta+\beta v+\beta \nabla u \cdot \eta\Big) \phi \, d\sigma\\
			=&\int_{\partial B_{R}} \Big(\frac{\partial v}{\partial \nu}+u_{rr}\zeta+\beta v+\beta u_{r} \zeta \Big) \phi \, d\sigma.
		\end{align*} 
		Since $\phi$ is arbitrary, we have $\frac{\partial v}{\partial\nu}+\beta v+\beta u_{r}\zeta+u_{rr}\zeta=0 $ on the $\partial B_{R}$. Also, it is clear that $-\Delta v=0$ in $B_{R}$.
		Hence $v$ satisfies
		\begin{align}
			\label{eqs:v-equations}
			\begin{cases}
				-\Delta v =0 & \textit{in $B_{R}$} \\
				\frac{\partial v}{\partial\nu}+\beta v+\beta u_{r}\zeta+u_{rr}\zeta=0 & \textit{on $\partial B_{R}$}.
			\end{cases}
		\end{align}
$\eqref{eqs:v-equations}_2$ can also be derived via differentiating the boundary condition of $u(t)$ on $\partial F_t(B_R)$.

		Now, we calculate the second shape derivative. Taking the derivative of
		\eqref{eq:Robin energy-first-derivative} again, we can obtain the second derivative of $J_{\beta}(F_{t}(B_R))$ at $t=0$.
		First, by \eqref{evolutionofarea2}-\eqref{evolutionofnormalspeed} we have
\begin{align*}
    \frac{d}{dt}\left(\eta(F_t(p))\cdot \nu(t)(p)\, d\sigma_t\right)=\left(\eta(F_t(p))\cdot \nu(t)(p)\right) \div \eta (F_t(p)) \, d\sigma_t.
\end{align*}
       Hence 
		\begin{align*}
			K_{1}:=&\frac{d}{dt}\Big|_{t=0}\int_{\partial F_{t}(B_{R})}-\beta^{2}u^{2}(t)\zeta \, d\sigma_{t}\\
			=&-\beta^{2}\int_{\partial B_{R}}\left(2u(v+<\nabla u,\eta>)\zeta+u^{2}\zeta \div\eta\right) \, d\sigma\\
			=&\int_{\partial B_{R}}\left(-2\beta^{2}uv\zeta-2\beta^{2}u\zeta<\nabla u,\eta>-\beta^{2}u^{2}\zeta \div\eta\right) \, d\sigma \\
			=&\int_{\partial B_{R}}\left(-2\beta^{2}uv\zeta-2\beta^{2}uu_{r}\zeta^{2}-\beta^{2}u^{2}\zeta \div\eta\right) \, d\sigma \\
			=&\int_{\partial B_{R}}\left(-2\beta^{2}u(v\zeta+u_{r}\zeta^{2})-\beta^{2}u^{2}\zeta \div\eta\right) \, d\sigma.
		\end{align*}
		In view of \eqref{eqs:v-equations}, similarly we have
		\begin{align*}	K_{2}:=&\frac{d}{dt}\Big|_{t=0}\int_{\partial F_{t}(B_{R})}\frac{1}{2}|\nabla u(t)|^{2}\zeta \, d\sigma_{t}\\
			=&\int_{\partial B_{R}} \left(u_{\nu}v_{\nu}\zeta+u_{\nu}u_{\nu\nu}\zeta^{2}+\frac{1}{2}|\nabla u|^{2}\zeta \div\eta \right) \, d\sigma \\
			=&\int_{\partial B_{R}} \left(u_{r}\zeta(v_{r}+u_{rr}\zeta)+\frac{1}{2}|\nabla u|^{2}\zeta \div\eta\right) \, d\sigma \\
			=&\int_{\partial B_{R}} \left(-u_{r}\zeta(\beta v+\beta u_{r}\zeta)+\frac{1}{2}|\nabla u|^{2}\zeta \div\eta \right) \, d\sigma \\
			=&\int_{\partial B_{R}} \left(\beta^{2}uv\zeta+\beta^{2}u_{r}u\zeta^{2}+\frac{1}{2}|\nabla u|^{2}\zeta \div\eta \right) \, d\sigma \\
			=&\int_{\partial B_{R}} \left(\beta^{2}u(v\zeta+u_{r}\zeta^{2})+\frac{1}{2}|\nabla u|^{2}\zeta \div\eta \right) \, d\sigma,
		\end{align*}
		where we used that
		\begin{equation*}
			\nabla u\cdot(\nabla^{2}u\eta)=u_{r}(\nabla^{2}u:\eta\otimes\nu)=u_{r}u_{rr}\zeta,
		\end{equation*}
		due to \eqref{Hessian}.
		
		Since $u_{rr}+\frac{n-1}{R}u_{r}+f=0$ and \eqref{eq:Robin equations}, $u_{rr}-\beta Hu+f=0$
		, where $H=\frac{n-1}{R}$ is the mean curvature on $\partial B_{R}$. This implies that $\beta Hu=u_{rr}+f$. Using this representation, we have
		\begin{align*}
			K_{3}:=&\frac{d}{dt}\Big|_{t=0}\int_{\partial F_{t}(B_{R})}\frac{\beta}{2}u^{2}(t)H(t)\zeta \, d\sigma_{t}\\
			=&\int_{\partial B_{R}} \left(\beta uvH\zeta+\beta uH\zeta<\nabla u,\eta>+\frac{\beta}{2}u^{2}\zeta(-\Delta_{\partial B_{R}}\zeta-\frac{n-1}{R^{2}}\zeta)+\frac{\beta}{2}u^{2}H\zeta \div\eta \right) \, d\sigma\\
			=&\int_{\partial B_{R}} \left(\beta uvH\zeta+\beta uu_{r}H\zeta^{2}+\frac{\beta}{2}u^{2}\zeta(-\Delta_{\partial B_{R}}\zeta-\frac{n-1}{R^{2}}\zeta)+\frac{\beta}{2}u^{2}H\zeta \div\eta \right) \, d\sigma\\
			=&\int_{\partial B_{R}} \left(\beta Hu(v\zeta+u_{r}\zeta^{2})+\frac{\beta}{2}u^{2}\zeta(-\Delta_{\partial B_{R}}\zeta-\frac{n-1}{R^{2}}\zeta)+\frac{\beta}{2}u^{2}H\zeta \div\eta \right) \, d\sigma\\
			=&\int_{\partial B_{R}} \left( (u_{rr}+f)(v\zeta+u_{r}\zeta^{2}) \, d\sigma+\int_{\partial B_{R}}\frac{\beta}{2}u^{2}\zeta(-\Delta_{\partial B_{R}}\zeta-\frac{n-1}{R^{2}}\zeta) \right) \, d\sigma+\int_{\partial B_{R}}\frac{\beta}{2}u^{2}H\zeta \div\eta \, d\sigma.
		\end{align*}
		In the end,
		\begin{align*}
			K_{4}:=&\frac{d}{dt}\Big|_{t=0}\int_{\partial F_{t}(B_{R})}-fu(t)\zeta \, d\sigma_{t}\\
			=&\int_{\partial B_{R}} \left(-(\nabla f\cdot\eta)u\zeta-fv\zeta-f<\nabla u,\eta>\zeta-fu\zeta \div\eta \right) \, d\sigma\\
			=&\int_{\partial B_{R}} \left(-f_{r}u\zeta^{2}-fv\zeta-fu_{r}\zeta^{2}-fu\zeta \div\eta \right) \, d\sigma\\
			=&\int_{\partial B_{R}} \left(-f_{r}u\zeta^{2}-f(v\zeta+u_{r}\zeta^{2})-fu\zeta \div\eta \right) \, d\sigma.
		\end{align*}
		In view that $F_{t}$ is volume-preserving with respect to $B_{R}$, by Proposition \ref{vp} we have	
		\begin{align*}
			\int_{\partial B_{R}}\zeta \, d\sigma=0 \quad \textit{and} \quad \int_{\partial B_{R}}\zeta \div\eta \,d\sigma=0.
		\end{align*}
		Summarizing the above equations, we have
		\begin{align*}
			\frac{d^{2}}{dt^{2}}\Big|_{t=0}J_{\beta}^{f}(F_{t}(B_{R}))&=K_{1}+K_{2}+K_{3}+K_{4}\\
			&=\frac{\beta}{2} \left(u(R)\right)^{2}\int_{\partial B_{R}} \left( -\Delta_{\partial B_R}\zeta-\frac{(n-1)}{R^{2}}\zeta \right)\zeta \,  d\sigma-\int_{\partial B_{R}}f_{r}u\zeta^{2} \, d\sigma\\
			&+\int_{\partial B_{R}}(v\zeta+u_{r}\zeta^{2})(u_{rr}-\beta^{2}u) \, d\sigma.
		\end{align*} 
		
	\end{proof}
	
	In order to understand that by imposing what condition on a positive radial function $f$ will ball be stable, we would like to express $u(R), u_{r}(R)$ and $u_{rr}(R)$ in terms of the information of $f$. Since $f$ is radial, by the divergence theorem, we have 
	\begin{equation*}
		-\beta u(R)\int_{\partial B_{R}} \, d\sigma=\int_{\partial B_{R}}-\beta u \, d\sigma=\int_{\partial B_{R}}\frac{\partial u}{\partial \nu} \, d\sigma=\int_{B_{R}}\Delta u \, dx=-\int_{B_{R}}f \, dx.
	\end{equation*}
	Then,
	\begin{equation}
		u(R)=\frac{R}{n\beta}\bar{f}(R).
		\label{eq:represent u(R)}
	\end{equation}
	By \eqref{eq:Robin equations}, we have 
	\begin{equation}
		u_{r}(R)=-\beta u(R)=-\frac{R}{n}\bar{f}(R),
		\label{eq:represent u_{r}(R)}
	\end{equation}
 and
	\begin{equation}
		u_{rr}(R)=-f(R)-\frac{n-1}{R}u_{r}(R)=-f(R)+\frac{n-1}{n}\bar{f}(R).
		\label{eq:represent u_rr(R)}
	\end{equation}
	
	Now we are ready to prove Theorem \ref{thm:Robin stable iff condition}.
	
	\begin{proof}[Proof of Theorem \ref{thm:Robin stable iff condition}]
		Let $\eta$ be a smooth velocity field of the volume preserving flow
		starting from $B_{R}$. Since the first eigenvalue of Laplacian on
		the unit sphere is $(n-1)$, it follows from Proposition \ref{pro:Robin-second variation} that
		
		\begin{equation}
			\frac{d^{2}}{dt^{2}}\Big|_{t=0}J_{\beta}^{f}(F_{t}(B_{R}))\geq\int_{\partial B_{R}}-f_{r}u\zeta^{2} \, d\sigma+\int_{\partial B_{R}}(v\zeta+u_{r}\zeta^{2})(u_{rr}-\beta^{2}u) \, d\sigma,
			\label{eq:fir-yh}
		\end{equation}
		where $u$ is the solution to \eqref{eq:Robin equations} with $\Omega=B_{R}$. By \eqref{eqs:v-equations},
		we have $\frac{\partial v}{\partial\nu}+\beta v=-(\beta u_{r}\zeta+u_{rr}\zeta)$ on $\partial B_{R}$,
		and so
		\begin{equation}
			\int_{\partial B_{R}}(\frac{\partial v}{\partial\nu}+\beta v) \, d\sigma=-\int_{\partial B_{R}}(\beta u_{r}\zeta+u_{rr}\zeta) \, d\sigma=0,
			\label{eq:sec-yh}
		\end{equation}
		where we used $\int_{\partial B_{R}}\zeta d\sigma=0$ and that $u$ is radial. By the divergence
		theorem, we have
		\begin{equation}
			\int_{\partial B_{R}}\frac{\partial v}{\partial\nu}d\sigma=\int_{B_{R}}\Delta vdx=0,
			\label{eq:thi-yh}
		\end{equation}
		where we have used \eqref{eqs:v-equations}. Comparing (\ref{eq:sec-yh}) with (\ref{eq:thi-yh}), we have 
		\begin{equation}
  \label{v0}
			\int_{\partial B_{R}}v \, d\sigma=0.
		\end{equation}
With this observation, we can now estimate $\int_{\partial B_{R}}v\zeta(u_{rr}-\beta^{2}u) \, d\sigma$.
		In view of \eqref{v0} and that the second Steklov eigenvalue
		over $B_{R}$ is $1/R$, we have 
		\begin{equation*}
			\int_{\partial B_{R}}v^{2} \, d\sigma\leq R\int_{B_{R}}|\nabla v|^{2} \, dx.
		\end{equation*}
		Hence
		\begin{align*}
			\int_{\partial B_{R}}v^{2} \, d\sigma\leq& R\int_{B_{R}}|\nabla v|^{2} \, dx\\
			=&R\int_{\partial B_{R}}\frac{\partial v}{\partial\nu}v \, d\sigma\\
			=&R\int_{\partial B_{R}}(-\beta v-\beta u_{r}\zeta-u_{rr}\zeta)v \, d\sigma\\
			=&R\int_{\partial B_{R}}-\beta v^{2} \, d\sigma+R\int_{\partial B_{R}}(\beta^{2}u-u_{rr})v\zeta \, d\sigma,
		\end{align*}
		where we have used \eqref{eq:Robin equations} and \eqref{eqs:v-equations}. Therefore, we have
		\begin{equation*}
			\int_{\partial B_{R}}v^{2} \, d\sigma\leq\frac{R}{1+\beta R}\int_{\partial B_{R}}(\beta^{2}u-u_{rr})v\zeta \, d\sigma.
		\end{equation*}
	Hence
		\begin{align*}
			\left(\int_{\partial B_{R}}(u_{rr}-\beta^{2}u)v\zeta \, d\sigma \right)^{2}\leq&\int_{\partial B_{R}}(u_{rr}-\beta^{2}u)^{2}\zeta^{2} \, d\sigma\int_{\partial B_{R}}v^{2} \, d\sigma\\
			\leq&\frac{R}{1+\beta R}\int_{\partial B_{R}}(u_{rr}-\beta^{2}u)^{2}\zeta^{2} \, d\sigma\int_{\partial B_{R}}(\beta^{2}u-u_{rr})v\zeta \, d\sigma.
		\end{align*}
	We now have the estimate for $\int_{\partial B_{R}}v\zeta(u_{rr}-\beta^{2}u) \, d\sigma$
		as below:
		\begin{equation}
			\int_{\partial B_{R}}(\beta^{2}u-u_{rr})v\zeta \, d\sigma\leq\frac{R}{1+\beta R}\int_{\partial B_{R}}(\beta^{2}u-u_{rr})^{2}\zeta^{2} \, d\sigma.
			\label{eq:fou-yh}
		\end{equation}
		Combining \eqref{eq:fir-yh} and \eqref{eq:fou-yh}, we have
		\begin{align*}
			\frac{d^{2}}{dt^{2}}\Big|_{t=0}J_{\beta}^{f}(F_{t}(B_{R}))\geq&\int_{\partial B_{R}}-f_{r}u\zeta^{2} \, d\sigma+\int_{\partial B_{R}}u_{r}\zeta^{2}(u_{rr}-\beta^{2}u) \, d\sigma+\int_{\partial B_{R}}v\zeta(u_{rr}-\beta^{2}u) \, d\sigma\\
			\geq&\int_{\partial B_{R}} \left(-f_{r}u+u_{r}(u_{rr}-\beta^{2}u)-\frac{R}{1+\beta R}(u_{rr}-\beta^{2}u)^{2}\right)\zeta^{2} \, d\sigma\\
			=&\int_{\partial B_{R}}\left(-f_{r}u+(u_{rr}-\beta^{2}u)(u_{r}-\frac{R}{1+\beta R}(u_{rr}-\beta^{2}u))\right)\zeta^{2} \, d\sigma.
		\end{align*}
		Using formulas \eqref{eq:represent u(R)}, \eqref{eq:represent u_{r}(R)} and \eqref{eq:represent u_rr(R)}, we have 
		\begin{align*}
			&\frac{d^{2}}{dt^{2}}\Big|_{t=0}J_{\beta}^{f}(F_{t}(B_{R}))\\
        \geq&\int_{\partial B_{R}}\frac{-R}{1+\beta R}\left( \left(f-\frac{n-1-R\beta}{n}\bar{f}(R) \right)(f-\bar{f}(R))+\frac{1+\beta R}{n\beta}f_{r}\bar{f}(R)\right)\zeta^{2} \, d\sigma.
		\end{align*}
		Therefore, if $f$ satisfies \eqref{eq: Robin stable iff condition}, then we have $\frac{d^{2}}{dt^{2}}\Big|_{t=0}J_{\beta}^{f}(F_{t}(B_{R}))\geq0$.
		
		Conversely, we suppose that $B_{R}$ is stable to $J_{\beta}^{f}(\cdot)$,
		then specially we have 
		\begin{equation*}
			\frac{d^{2}}{dt^{2}}\Big|_{t=0}J_{\beta}^{f}(F_{t}(B_{R}))\geq0,
		\end{equation*}
		for $F_{t}(B_{R})=\{x+t\eta:x\in B_{R}\}$ where $\eta$ is a nonzero constant
		vector field. Let $\eta=(c_{1},c_{2},\dots,c_{n})^{T}$, and thus 
		\begin{equation*}
			\zeta=\frac{1}{R}\sum_{i=1}^{n}c_{i}x_{i}.
		\end{equation*}
		With this choice of $F_{t}$, $\zeta$ is now the first eigenfunction
		of Laplacian on $\partial B_{R}$. Hence, we have 
		\begin{equation*}
			\frac{d^{2}}{dt^{2}}\Big|_{t=0}J_{\beta}^{f}(F_{t}(B_{R}))=\int_{\partial B_{R}}-f_{r}u\zeta^{2} \, d\sigma+\int_{\partial B_{R}}(v\zeta+u_{r}\zeta^{2})(u_{rr}-\beta^{2}u) \, d\sigma.
		\end{equation*}
		In this choice of $\eta$ and from \eqref{eqs:v-equations}, $v$ must be of the form
		\begin{equation}
		v(x)=a\sum_{i=1}^{n}c_{i}x_{i}+b.
			\label{eq:v-zeta}
		\end{equation}
		From $\eqref{eqs:v-equations}_2$, we have 
		\begin{align*}
			0=&\frac{\partial v}{\partial\nu}+\beta v+\beta u_{r}\zeta+u_{rr}\zeta\\
			=&a\zeta+\beta(aR\zeta+b)+(\beta u_{r}+u_{rr})\zeta\\
			=&(a+a\beta R+\beta u_{r}+u_{rr})\zeta+\beta b
		\end{align*}
		on $\partial B_{R}$. Therefore, $v$ is a solution to \eqref{eqs:v-equations} if and only if
		\begin{equation}
			\begin{cases}
				a+a\beta R+\beta u_{r}(R)+u_{rr}(R)=0 \\ b=0
			\end{cases}
			\label{eq:v-coeff}
		\end{equation}
		Now we have 
		\begin{align*}
			\frac{d^{2}}{dt^{2}}\Big|_{t=0}J_{\beta}^{f}(F_{t}(B_{R}))=&\int_{\partial B_{R}}-f_{r}u\zeta^{2} \, d\sigma+\int_{\partial B_{R}}(v\zeta+u_{r}\zeta^{2})(u_{rr}-\beta^{2}u) \, d\sigma\\
			=&\int_{\partial B_{R}} \left(-f_{r}u+(aR+u_{r})(u_{rr}-\beta^{2}u)\right)\zeta^{2} \, d\sigma\\
			=&\int_{\partial B_{R}} \left(-f_{r}u+(u_{rr}-\beta^{2}u)(u_{r}-\frac{R}{1+\beta R}(u_{rr}-\beta^{2}u))\right)\zeta^{2} \, d\sigma,
		\end{align*}
		where we have used (\ref{eq:v-zeta}) and (\ref{eq:v-coeff}). Using \eqref{eq:represent u(R)}, \eqref{eq:represent u_{r}(R)}
		and \eqref{eq:represent u_rr(R)}, we have
		\begin{equation*}
			\frac{d^{2}}{dt^{2}}\Big|_{t=0}J_{\beta}^{f}(F_{t}(B_{R}))=\int_{\partial B_{R}}\frac{-R}{1+\beta R} \left( \left(f-\frac{n-1-R\beta}{n}\bar{f}(R) \right)(f-\bar{f}(R))+\frac{1+\beta R}{n\beta}f_{r}\bar{f}(R)\right)\zeta^{2} \, d\sigma.
		\end{equation*}
		Since $\left(f-\frac{n-1-R\beta}{n}\bar{f}(R)\right)(f-\bar{f}(R))+\frac{1+\beta R}{n\beta}f_{r}\bar{f}(R)$
		is a constant on the boundary $\partial B_{R}$, the stability of $B_R$ immediately implies that
		\begin{equation*}
			\left(f(R)-\frac{n-1-R\beta}{n}\bar{f}(R)\right)(f(R)-\bar{f}(R))+\frac{1+\beta R}{n\beta}f_{r}(R)\bar{f}(R)\leq0.
		\end{equation*}
	\end{proof}

\begin{remark}
\label{translation}
    From the proof above, we immediately have that if \eqref{eq: Robin stable iff condition} is not satisfied, then $J^f_\beta(\cdot)$ strictly increases after a small translation of $B_R$. This provides the precise direction of variation along which $B_R$ is not stable to $J^f_\beta(\cdot)$ when \eqref{eq: Robin stable iff condition} is violated.
\end{remark}  

\begin{remark}
\label{rmk-iff}
The condition \eqref{eq: Robin stable iff condition} also says that the stability of ball shape not only depends on both the dimension $n$ and the size of $\beta$, but also depends on the size of the ball. See also the fourth statement in Corollary \ref{cor1}.
\end{remark}

\vskip 0.2cm
At the end of the section, we give some other remarks. Given a $C^{2}$ shape functional $\mathcal{E}(\cdot)$, we say that $B_{R}$ is \textit{strictly stable} to $\mathcal{E}(\cdot)$, if $\frac{d^{2}}{dt^{2}}\big|_{t=0}\mathcal{E}(F_{t}(B_{R}))$ is strictly positive for any smooth volume preserving flow map $F_{t}$. For example, $B_{R}$ is strictly stable for $J^{f}_{\beta}(\cdot)$ when the inequality \eqref{eq: Robin stable iff condition} is strict.

Dambrine computed the second order variation formula for $J^{f}_{\infty}(\cdot)$ in \cite{Dam}, which also provides the stability condition for $J^{f}_{\infty}(\cdot)$ at $B_{R}$. We also refer to his paper with Pierre \cite{DP00}. Their results state that when $f$ is a radial function about the origin, $B_{R}$ is stable to $J_{\infty}^{f}(\cdot)$ if and only if $f(R)\leq\bar{f}(R)$. When this inequality for $f$ is strict, then $B_{R}$ is strictly stable to $J^{f}_{\infty}(\cdot)$.  

By the result of Dambrine-Pierre and a closer look at \eqref{eq: Robin stable iff condition}, the stability conditions for $J_{\infty}^f(\cdot)$ and $J_{\beta}^f(\cdot)$ at $B_R$ can be related via the following corollary.
\begin{corollary}
	\label{thm:limit&general_ball_stable}
	Let $f>0$ be a smooth radial function about the origin, and $B_{R} \subset \mathbb{R}^n$ be the ball whose radius is R centered at the origin. 
	\begin{enumerate}
	    \item If $B_{R}$ is strictly stable to $J_{\infty}^{f}(\cdot)$, then $B_{R}$ is also strictly stable to $J_{\beta}^{f}(\cdot)$ as $\beta$ sufficient large.
	   \item  If $B_{R}$ is not stable to $J_{\infty}^{f}(\cdot)$, then $B_{R}$ is also unstable to $J_{\beta}^{f}(\cdot)$ as $\beta$ sufficient large.
	\end{enumerate}		
\end{corollary}

\section{Global results on optimality of the ball and counterexample of Robin Talenti pointwise comparison}
In this section, we will prove three global optimization results, and we will also provide a counterexample to the Robin Talenti pointwise comparison results as a consequence.

First, recall that in the previous section, our stability results indicate that for any $\beta>0$, when $R<\tfrac{n-1}{\beta}$, there always exists a radially decreasing function $f$ such that $B_R$ is not a local minimizer to $J^f_\beta(\cdot)$ among domains with the same volume as $|B_R|$.  However, when $\beta=+\infty$, we will show that for any $R>0$ and any radially decreasing function $f$, $B_R$ is always a global minimizer to $J^f_\infty(\cdot)$ among domains with the same volume as $|B_R|$, as stated in the following proposition.


\begin{proposition}
\label{thm:global}
Let $\Omega\subset \R^{n}$ be a bounded domain with Lipschitz boundary and $f>0$ be a radially decreasing function about the origin. Then 
\begin{align}
\label{eq:global}
J_{\infty}^{f}(\Omega^{\sharp})\leq J_{\infty}^{f}(\Omega),
\end{align}
where $\Omega^{\sharp}$ is a ball centered at the origin with the same volume as $|\Omega|$.
\end{proposition}

Second, we consider the maximization of the averaged temperature functional. Different from the stability breaking results as stated in Corollary \ref{cor1} in the case of maximizing the averaged heat, the next result shows that for any radially decreasing function $f$, any $\beta>0$, and any volume constraints, the ball centered at the maximal point of $f$ is always the maximizer to the averaged temperature, at least in two dimensions.

\begin{proposition}
\label{thm:global2}
Let $\Omega\subset \R^{n}$ be a bounded domain with Lipschitz boundary and $f>0$ be a radially decreasing function about the origin. Let $u_\Omega^{f,\beta}$  and $u_\Omega^f$ be the solutions to \eqref{eq:Robin equations} and \eqref{Dirichletcondition}, respectively. Then, for the shape functionals $A_{\beta}^{f}(\Omega):=\int_{\Omega}u_{\Omega}^{f,\beta}dx$ and $A_{\infty}^{f}(\Omega):=\int_{\Omega}u_{\Omega}^{f}dx$, we have:
\begin{align}
\label{2dimresult}
A_{\beta}^{f}(\Omega)\leq A_{\beta}^{f}(\Omega^{\sharp}), \quad \mbox{for $n=2$ and $\forall\, \beta\in (0,+\infty)$}
\end{align}
and 
\begin{align}
A_{\infty}^{f}(\Omega)\leq A_{\infty}^{f}(\Omega^{\sharp}), \quad \mbox{for all $n\geq 2$},
\end{align}
where $\Omega^{\sharp}$ is the ball centered at the origin with the same volume as that of $\Omega$.
\end{proposition}

Third, we will obtain a similar result for a thin insulation shape optimization problem. All of these results are based on applications of the Talenti type results and the following simple observation.

\begin{lemma}
		\label{lemma}
		Suppose that $f>0$ is a radially decreasing function about the origin defined in $\mathbb{R}^n$, and $B_{R}\subset\mathbb{R}^n$ is a ball of radius $R$ centered at the origin. Then, for any measurable set $\Omega\subset \mathbb{R}^n$ satisfying $|\Omega|=|B_{R}|$, we have
		\begin{align}
			(f\Big|_{\Omega})^{\sharp}(x)\leq f(x)
		\end{align}
		for any $x$ in $B_{R}$, where $f\Big|_{\Omega}$ is the restriction of $f$ on $\Omega$ and $(f\Big|_{\Omega})^{\sharp}$ is the spherical decreasing rearrangement of $(f\Big|_{\Omega})$ with respect to $\Omega$.
	\end{lemma}
To prove the lemma, let us recall the basic notion of spherical decreasing rearrangement. Let $\omega$ be a measurable set in $\mathbb{R}^n$, and we denote by $\omega^{\sharp}$ the ball of the same volume as that of $\omega$ centered at the origin. Suppose that $\varphi(x)$ is a non-negative measurable function defined on a measurable set $\omega\subset \mathbb{R}^n$, then the distribution function $\mu(t)$ is defined as $\mu(t)=|\{ x\in \omega: |\varphi(x)|>t \}|$, the decreasing rearrangement of $\varphi(x)$ with respect to $\omega$ is defined as $\varphi^{*}(s)=\inf \{t\geq0: \mu(t)<s \}$, and the spherically decreasing rearrangement is defined as $\varphi^{\sharp}(x)=\varphi^{*}(C_n |x|^n)$, where $C_n$ is the measure of the unit ball in $\mathbb{R}^n$. For properties of spherically symmetric rearrangement, we refer to Hardy \cite{Hardy} and Poly\'a-Szeg\"o \cite{PS}. 
	
	\begin{proof}[Proof of Lemma \ref{lemma}]
		Let $\mu_{\Omega}(t)=|\{x\in \Omega: f(x)>t\}|$, and $\mu(t)=|\{x\in B_{R}: f(x)>t\}|$. By definition, we have
		\begin{align}
			(f\Big|_{\Omega})^{\sharp}(x):=\inf \{t>0:\mu_{\Omega}(t)<C_{n}|x|^n\}
		\end{align}
	and 
	\begin{align}
		f(x)= \inf\{t>0:\mu(t)<C_{n}|x|^n\},
	\end{align}
	where $C_n$ is the volume of the unit ball in $\mathbb{R}^n$.
	In view that $\{x\in \Omega: f(x)>t\}$ is a subset of $\{x\in \mathbb{R}^n: f(x)>t\}$, we have that $\mu_{\Omega}(t)\leq \mu(t)$ for all $t>0$. 
	Therefore for any $x\in B_{R}$, we have $\inf \{t>0:\mu_{\Omega}(t)<C_{n}|x|^n\} \leq \inf\{t>0:\mu(t)<C_{n}|x|^n\}$. 
	Hence $(f\Big|_{\Omega})^{\sharp}(x)\leq f(x)$ for all $x$ in $\Omega^{\sharp}=B_{R}$.
	\end{proof}	
	
Before proving Proposition \ref{thm:global} and Proposition \ref{thm:global2}, we recall the Talenti's comparison results. Let $\Omega\subset \R^{n}$ be a bounded domain with Lipschitz boundary and $f$ be a smooth function defined in $\Omega$. Talenti's pointwise comparison result for the Dirichlet condition was proved by Talenti in his famous paper \cite{Talenti}, stating that 
\begin{align*}
(u_{\Omega}^{f\big|_{\Omega}})^{\sharp}\leq u_{\Omega^{\sharp}}^{(f\big|_{\Omega})^{\sharp}}, \quad \mbox{for all $x\in \Omega^{\sharp}$}.
\end{align*}
	
	There are at least two proofs of this result which are available. The original proof by Talenti \cite{Talenti} uses the classical isoperimetric inequality, involving the De Giorgi perimeter of $\Omega$. Another proof by Lions \cite{Lions} uses a differential inequality between the distribution functions of $u$ and $v$. The rigidity problem of Talenti's comparison result with Dirichlet boundary condition was also studied by Alvino, Lions and Trombetti \cite{AP}, and Kesavan \cite{Kesavan}. They proved that if $(u_{\Omega}^{f\big|_{\Omega}})^{\sharp}=u_{\Omega^{\sharp}}^{(f\big|_{\Omega})^{\sharp}}$ for any $x$ in $\Omega^{\sharp}$, then the domain $\Omega=\Omega^{\sharp}+x_0$, where $x_0$ is a fixed point in $\mathbb{R}^n$. 

Recently, Talenti-type comparison results have been successfully extended by Alvino, Nitsch and Trombetti to the case of Robin boundary conditions, see \cite{AN}. After this pioneering work, Robin Talenti-type comparison results have been extensively studied, including (but not limited to) \cite{ACNT, AL25, AP, BS25, CLW23, CNT24} and the references therein. For Robin boundary conditions, the validity of Talenti-type comparisons is sensitive to the spatial dimension, and several related open problems remain open, especially the pointwise comparison for $f \equiv 1$ in dimensions larger than or equal to $3$. Therefore, as an application of \cite{AN}, we can only prove \eqref{2dimresult} in two dimensions.

\begin{proof}[Proof of Proposition \ref{thm:global}]
Using the Hardy-Littlewood inequality and Talenti's comparison result for Dirichlet boundary condition, we have
\begin{align}
\label{eq:pfglobal1}
\int_{\Omega}u_{\Omega}^{f\big|_{\Omega}}fdx\leq \int_{\Omega^{\sharp}}(u_{\Omega}^{f\big|_{\Omega}})^{\sharp}(f\big|_{\Omega})^{\sharp}dx\leq \int_{\Omega^{\sharp}}u_{\Omega^{\sharp}}^{(f\big|_{\Omega})^{\sharp}}(f\big|_{\Omega})^{\sharp}dx.
\end{align}

From the Lemma \ref{lemma} and the maximum principle, we have
\begin{align}
\label{eq:pfglobal2}
\int_{\Omega^{\sharp}}u_{\Omega^{\sharp}}^{(f\big|_{\Omega})^{\sharp}}(f\big|_{\Omega})^{\sharp}dx\leq \int_{\Omega^{\sharp}}u_{\Omega^{\sharp}}^{f}fdx.
\end{align}

Combining with \eqref{eq:pfglobal1} and \eqref{eq:pfglobal2}, we have
\begin{align*}
-2J_{\infty}^{f}(\Omega)\leq \int_{\Omega}fu_{\Omega}^{f}dx\leq \int_{\Omega^{\sharp}}fu_{\Omega^{\sharp}}^{f}dx=-2J_{\infty}^{f}(\Omega^{\sharp}).
\end{align*}
We complete the proof.
\end{proof}

\begin{proof}[Proof of Proposition \ref{thm:global2}]
By properties of rearrangement and the Robin-Laplacian Talenti-type pointwise comparison results, see \cite[Theorem 1.1]{AN}, we have that when $n=2$,
\begin{align}
\label{attention}
    \int_{\Omega} u_{\Omega}^{f,\beta}\, dx = \int_{\Omega^\sharp} (u_{\Omega}^{f,\beta})^{\sharp}\, dx\le \int_{\Omega^\sharp} u_{\Omega^{\sharp}}^{(f\big|_{\Omega})^{\sharp},\beta}\, dx,
\end{align}
Using Lemma \ref{lemma} and the maximum principle, we have
\begin{align}
\label{attention2}
\int_{\Omega^\sharp} u_{\Omega^{\sharp}}^{(f\big|_{\Omega})^{\sharp},\beta}\, dx\leq \int_{\Omega^{\sharp}}u_{\Omega^{\sharp}}^{f,\beta}dx.
\end{align}
Combining with \eqref{attention} and \eqref{attention2}, we have
\begin{align*}
\int_{\Omega}u_{\Omega}^{f,\beta}dx\leq \int_{\Omega^{\sharp}}u_{\Omega^{\sharp}}^{f,\beta}dx.
\end{align*}

For the case $\beta=\infty$, since the Dirichlet-Laplacian Talenti's comparison results are true in any dimensions, a similar argument as above implies that 
\begin{align*}
    \int_{\Omega}u_{\Omega}^{f}dx\leq \int_{\Omega^{\sharp}}u_{\Omega^{\sharp}}^{f}dx.
\end{align*}
\end{proof}

\begin{remark}
Theorem \ref{thm:global2} admits a physical interpretation: in conduction heat transfer in any dimension, and in convection heat transfer in two dimensions, if the heat source is radially decreasing and concentrated around a single point, then the optimal way to maximize the averaged temperature is to choose the domain as a ball and place the heat source at its center. Whether or not the ball remains the optimal shape for convection heat transfer in higher dimensions remains an open question.
\end{remark}

Talenti-type results can be applied to consider another shape optimization problem from thin insulation background. For the physical background and mathematical derivations, we refer to \cite{Buttazzo}, \cite{BBN} and \cite{BBN1}, and see also \cite{HLL1}.

Here we briefly recall the setting. In this thin insulation problem introduced in \cite{Buttazzo}, the width of the insulating layer is assumed to be comparable to the heat transmission coefficient along the inner layer, the heat conduction transfer mode is dominant along the outer layer, and the total amount of insulation material is fixed. The heat source function is denoted by $f$.  For a given amount of insulating material $m$ and given volume of the conductor $|\Omega|$, the aim is to determine the shape of $\Omega$, and the related optimal distribution of insulating material $h$ with mass constraint $\int_{\partial \Omega}hd\sigma=m$, in order to maximize the heat content. Here $m$ is a fixed constant.

In other words, for a fixed $m>0$, the aim is to find out the shape of $\Omega$, which maximizes $A(\Omega):=\max_{h\in \mathcal{H}_m}\int_{\Omega}u_hdx$, where $u_h$ solves 
\begin{align}
    \label{uh}
\begin{cases}
-\Delta u=f\quad &\mbox{in $\Omega$}\\
h\frac{\partial u}{\partial \nu}+ u=0\quad &\mbox{on $\partial \Omega$}
\end{cases}    
\end{align}
and 
$$\mathcal{H}_m:=\{h \in L^1(\partial \Omega): h\ge 0 ,\, \int_{\partial \Omega}h \, d\sigma=m\}.$$

It is proved in \cite{DNST} that when $f \equiv 1$, balls uniquely maximize $A(\cdot)$ among domains with fixed volume. Here, as an application of  Lemma \ref{lemma} and the Talenti-type comparison result in \cite{ACNT}, we are able to slightly extend the result in \cite{DNST} to the case of radially decreasing heat sources in two dimensions.
\begin{proposition}
\label{thermaltheorem}
    Let $f$ be a radially decreasing function about the origin. Then among domains with fixed volume, in two dimensions, the ball centered at the origin is a maximizer to $A(\cdot)$.
\end{proposition}

\begin{proof}
For $h\in \mathcal{H}_m$, we consider the perturbed equation
\begin{align}
    \label{uhl}
\begin{cases}
-\Delta u=f\quad &\mbox{in $\Omega$}\\
\frac{\partial u}{\partial \nu}+\frac{\lambda}{1+\lambda h} u=0\quad &\mbox{on $\partial \Omega$},
\end{cases}       
\end{align}and we let $u_{h}^\lambda$ be the solution to \eqref{uhl}. By the Talenti-type $L^1$ comparison results for general boundary conditions proved in \cite{ACNT}, for any such $h$, we have:
\begin{align}
\label{acnttheorem}
    \int_{\Omega}u_h^\lambda \,dx \le \int_{\Omega^\sharp} v^\lambda \, dx,
\end{align}where $v^\lambda$ is the solution to 
\begin{align}
    \label{betabar}
\begin{cases}
-\Delta v=(f\Big|_{\Omega})^\sharp \quad &\mbox{in $\Omega^\sharp$}\\
\frac{\partial v}{\partial \nu}+\bar{\beta}^\lambda v=0\quad &\mbox{on $\partial \Omega^\sharp$},
\end{cases}
\end{align}with $\bar{\beta}^\lambda$ being a constant given by \begin{align}
    \label{betabarlimit}
\frac{P(\Omega^\sharp)}{\bar{\beta}^\lambda}=\int_{\partial \Omega}\frac{1+\lambda h}{\lambda}\, d\sigma\rightarrow \int_{\partial \Omega}h\, d\sigma=m,\quad \mbox{as $\lambda \rightarrow \infty$.}    
\end{align}
From \eqref{acnttheorem}-\eqref{betabarlimit}, the limit behavior of $u^\lambda_h$ and $v^\lambda$, Lemma \ref{lemma} and the maximum principle, we have
\begin{align}
    \label{final}
\int_{\Omega}u_h \, dx\le  \int_{\Omega^\sharp} w\, dx,
\end{align}where $w$ is the solution to 
\begin{align}
    \label{wsolution}
\begin{cases}
-\Delta w=f \quad &\mbox{in $\Omega^\sharp$}\\
\frac{\partial w}{\partial \nu}+\frac{1}{m}P(\Omega^\sharp) w=0\quad &\mbox{on $\partial \Omega^\sharp$}.
\end{cases}
\end{align}
Such $w$ is exactly the solution to \eqref{uh} with $\Omega=\Omega^\sharp$ and $h=m/P(\Omega^\sharp)$, which is a constant function belonging to $\in \mathcal{H}_m$. Hence we conclude that the supremum of $A(\cdot)$ is attained at $\Omega=\Omega^\sharp$ and the constant function belonging to $\mathcal{H}_m$ along the boundary is the optimal distribution in order for $A(\cdot)$ to achieve its supremum.
\end{proof} 

\vskip 0.3cm

We also mention that as a consequence of Theorem \ref{thm:Robin stable iff condition} and the proof of Proposition \ref{thm:global}, we immediately have an example illustrating that the Talenti-type pointwise comparison result for Robin boundary conditions fails for $\beta$ lying in a finite open interval, even if the source is radially decreasing and the domain under consideration is a small translation of a ball centered at the origin.

Indeed, according to (2)-(3) in Corollary \ref{cor1} and Remark \ref{translation}, there exists a pair $(f,B_{R})$ such that when the domain $\Omega$ is a minor translation of $B_R$, $J_\beta^{f}(\Omega)<J_\beta^{f}(B_R)$ for $\beta$ lying in a finite open interval. However, if the Robin Talenti pointwise comparison results were true for such translated domain $\Omega$ and for the $\beta$ lying in the range, then the similar argument as in the proof of Proposition \ref{thm:global} would imply that $J_{\beta}^{f}(\Omega)\ge J_\beta^{f} (B_R)$. This leads to a contradiction. 

\vskip 0.3cm

Our counterexample above is essentially from a variational point of view, which is quite different from the counterexample in \cite{AN}, which addresses a particular case $\beta=1/2$.

At the end of the section, we further discuss the example in \cite{AN}. Let $\Omega\subset\R^{2}$ be the union of two disks $D_{1}$(centered at the origin) and $D_{\epsilon}$ with radii $1$ and $\epsilon$, respectively. Let $\beta>0$ be a constant, and we set $f=1$ on $D_{1}$ and $f=0$ on $D_{\epsilon}$. 

We have
\begin{align*}
u_{\Omega}^{f,\beta}(x)=
\begin{cases}
0, \quad &x\in D_{\epsilon}\\
\frac{1}{2\beta}+\frac{1}{4}-\frac{1}{4}|x|^{2}, &x\in D_{1}
\end{cases}
\end{align*}
and 
\begin{align*}
u_{\Omega^{\sharp}}^{f^{\sharp},\beta}(x)=
\begin{cases}
\frac{1}{2\beta c}+\frac{1}{2}\ln c +\frac{1}{4}-\frac{1}{4}|x|^{2}, &|x|\le 1\\
-\frac{1}{2}\ln|x|+\frac{1}{2\beta c}+\frac{1}{2}\ln c, \quad &1<|x|<c,
\end{cases},
\end{align*}
where $\Omega^{\sharp}=B_{c}$ is a disk centered at the origin with radius $c=\sqrt{1+\epsilon^{2}}$. Hence
\begin{align*}
\Vert u_{\Omega^{\sharp}}^{f^{\sharp},\beta}\Vert _{L^{\infty}(\Omega^{\sharp})}-\Vert u_{\Omega}^{f,\beta}\Vert_{L^{\infty}(\Omega)}=\frac{1}{2\beta}\left(\frac{1}{c}-1\right)+\frac{1}{2}\ln c.
\end{align*}
Let $$\beta_0=\frac{1-\frac{1}{c}}{\ln c}.$$
When $\beta<\beta_0$, then $\Vert u_{\Omega^{\sharp}}^{f^{\sharp},\beta}\Vert _{L^{\infty}(\Omega^{\sharp})}<\Vert u_{\Omega}^{f,\beta}\Vert_{L^{\infty}(\Omega)}=\Vert (u_{\Omega}^{f,\beta})^\#\Vert_{L^{\infty}(\Omega)}$, and thus the Robin $L^\infty$ comparison fails for this range of $\beta$.

Therefore, motivated by the third statement in Corollary \ref{cor1} and the above example, it would be interesting to answer the following question:

\begin{question}
    Given a generic bounded Lipschitz domain $\Omega\subset \mathbb{R}^n$, does there exist a constant $\beta_{0}>0$ depending on $\Omega$, such that when $\beta>\beta_0$,
$$\Vert u_{\Omega^{\sharp}}^{f^{\sharp},\beta}\Vert _{L^{\infty}(\Omega^{\sharp})}\ge \Vert u_{\Omega}^{f,\beta}\Vert_{L^{\infty}(\Omega)}$$
for any $f \in L^2(\Omega)$, while when $\beta<\beta_0$, there exists $f\in L^2(\Omega)$ such that the above inequality fails? 
\end{question}

	\section{appendix: proof of Proposition \ref{geometricevolution}}
	
	\begin{proof}[Proof of Proposition \ref{geometricevolution}]
		Given $p \in M$, we let $(x^1, x^2, \cdots, x^{n-1})$ be local coordinates of $M$ near $p$, and thus they also serve as local coordinates of $M_t$ near $p(t):=F_t(p) \in M_t$. When taking derivatives of a vector field $\eta$, at $p(t)$ along $M_t$ with respect to the same local coordinates, we use the abbreviation that 
		\begin{align}
			\label{convention}
			\partial_\alpha \eta(p(t)):=\frac{\partial}{\partial x^\alpha}(\eta \circ F_t(p))=\left(\nabla \eta(p(t))\right) \partial_\alpha F_t(p),
		\end{align}where $\partial_\alpha F_t(p):=(\nabla F_t(p)) \frac{\partial p}{\partial x^\alpha}$ serves as the basis of $T_{p(t)}M_t$,\, $1\le \alpha \le n-1$. Let $g(t)(p)$ be the metric tensor on $M_t$ induced by $F_t$ at $p(t)$, and we write $g_{\alpha\beta}(t)(p)=\partial_{\alpha}F_t(p)\cdot \partial_\beta F_t(p)$. Then we have
		\begin{align}
			\label{gderi}
			g_{\alpha\beta}'(t)(p)=\frac{d}{dt}<\partial_{\alpha}F_t(p),\partial_{\beta}F_t(p)>=\partial_{\alpha} \eta(p(t)) \cdot \partial_{\beta} F_t(p)+\partial_{\alpha}F_t(p)\cdot \partial_{\beta}\eta(p(t)).
		\end{align}
		Let $|g(t)|$ be the determinant of the $(n-1)\times (n-1)$ matrix $(g_{\alpha\beta}(t))$. Let $g^{\alpha \beta}$ be the inverse matrix of $g_{\alpha \beta}$. Since \begin{align*}
			\nabla^ {M_t}\eta(p(t))=g^{\alpha \beta}(t)(p)\,\partial_\alpha \eta(p(t)) \otimes \partial_\beta F_t(p),
		\end{align*}
		\eqref{gderi} implies that
		\begin{align*}
			&\frac{d}{dt}\sqrt{|g(t)(p)|}\nonumber\\
            =&\frac{1}{2\sqrt{|g(t)(p)|}}|g(t)(p)|g^{\alpha\beta}(t)(p) \left(\partial_{\alpha} \eta(p(t)) \cdot \partial_{\beta} F_t(p)+\partial_{\alpha}F_t(p)\cdot \partial_{\beta}\eta(p(t))\right)\nonumber\\=&(\div_{M_t}\eta(p(t)))\sqrt{|g(t)(p)|}. 
		\end{align*}
		Since at $p(t)$ one has $d\sigma_t=\sqrt{|g(t)|}dx^1\cdots dx^{n-1}$, we immediately have \eqref{evolutionofarea2}.
		
		Next, let $\eta^T(p(t))$ be the tangential component of $\eta(p(t))$ at $T_{p(t)}M_t$. Then we have
		\begin{align}
			\label{etaT}
			\eta^T(p(t))=g^{\alpha\beta}(t)(p)<\eta(p(t)),\partial_{\beta}F_t(p)>\partial_{\alpha}F_t(p).
		\end{align}Therefore, since $\nu_p'(t)\in T_{p(t)}M_t$, where $\nu_p(t):=\nu(t)(p)$ is the unit normal to $T_{p(t)}M_t$ at $p(t)$, by \eqref{convention} and \eqref{etaT} we have
		\begin{align*}
			&\frac{d}{dt}\left(\eta(p(t))\cdot \nu_p(t)\right)\\
			=&\Big((\nabla \eta(p(t))) \eta(p(t))\Big)\cdot \nu_p(t)+<\eta(p(t)), g^{\alpha\beta}(t)(p)\Big( \nu_p'(t)\cdot \partial_{\alpha}F_t(p)\Big) \partial_{\beta} F_t(p)>\\
			=&\Big((\nabla \eta(p(t))) \eta(p(t))\Big)\cdot \nu_p(t)-g^{\alpha\beta}(t)(p)<\nu_p(t),\partial_{\alpha} \eta(p(t))><\eta(p(t)),\partial_{\beta} F_t(p)>\\
			=&\Big((\nabla \eta(p(t))) \eta^T(p(t))\Big)\cdot \nu_p(t)+\Big((\nabla \eta (p(t))) ((\eta(p(t))\cdot \nu_p(t))\nu_p(t) ) \Big)\cdot \nu_p(t)\\
			&\quad -<\nu_p(t),(\nabla \eta(p(t))) \eta^T(t)>\\
			=&\Big((\nabla \eta(p(t))) (\eta(p(t))\cdot \nu_p(t))\nu_p(t) ) \Big)\cdot \nu_p(t)\\
			=&(\eta(p(t))\cdot\nu_p(t)) \Big(\div \eta(p(t)) -\div_{M_t}\eta(p(t))\Big),
		\end{align*}where the last equality is obtained by taking the trace of the following identity:
		\begin{align*}
			\nabla \eta (p(t)) =\nabla ^{M_t}\eta(p(t)) +\nabla \eta(p(t)) \nu_p(t)\otimes \nu_p(t).
		\end{align*}
		Hence \eqref{evolutionofnormalspeed} is proved.
		
		Next, we compute the derivative of $h(t)$ at $p(t)$. Since
        \begin{align*}
            h_{\alpha \beta}(t)(p)=\partial_\alpha \nu_p(t)\cdot \partial_\beta F_t(p),
        \end{align*}
        using the notation \eqref{convention}, we have
		\begin{align*}
			h_{\alpha\beta}'(t)(p)=&\partial_\alpha\nu_p'(t)\cdot \partial_\beta F_t(p)+\partial_\alpha\nu_p(t)\cdot \partial_\beta\eta(p(t))\\
			=&\partial_\alpha(\nu_p'(t)\cdot \partial_\beta F_t(p))-\nu_p'(t)\cdot \partial_\alpha\partial_\beta F_t(p)+\partial_\alpha \nu_p(t)\cdot \partial_\beta\eta(p(t))\\
			=&-\partial_\alpha(\nu_p(t)\cdot \partial_\beta\eta(p(t)))-g^{\gamma \zeta}(t)(p)(\nu_p'(t)\cdot\partial_\gamma F_t(p))<\partial_\zeta F_t(p),\partial_\alpha\partial_\beta F_t(p)>
            \\ &+\partial_\alpha\nu_p(t)\cdot \partial_\beta \eta(p(t)),\quad \mbox{where $\gamma, \zeta=1,\cdots, n-1$}\\
			=&-<\nu_p(t),\partial_\alpha\partial_\beta\eta(p(t)>-g^{\gamma \zeta}(t)(\nu_p'(t)\cdot\partial_\gamma  F_t(p)) \Gamma_{\alpha \beta}^\delta(t)(p)g_{\zeta \delta}(t)(p), \\ &\mbox{where $\Gamma_{\alpha\beta }^\delta(t)(p)$ are Christoff symbols on $M_t$ at $p(t)$}\\
			=&-<\nu_p(t),\partial_\alpha\partial_\beta\eta(p(t))>+\Gamma_{\alpha\beta}^\gamma(t)(p)<\nu_p(t),\partial_\gamma\eta(p(t))>\\
			=&-<\nabla_\alpha\nabla_\beta\eta(p(t)),\nu_p(t)>.
		\end{align*}Hence \eqref{hijevolution} is proved.
		
		Last, to prove \eqref{evoonsphere}, we first note that by \eqref{gderi} and \eqref{hijevolution}, we have
		\begin{align*}
			& H'(t)(p)\\=&\frac{d}{dt}\left(g^{\alpha\beta}(t)(p) h_{\alpha \beta}(t)(p)\right)\\
			=&-2g^{\alpha \gamma}(t)(p) g^{\delta \beta}(t)(p)<\partial_\delta\eta(p(t)),\partial_\gamma F_t(p)>h_{\alpha\beta}(t)(p)-g^{\alpha \beta}(t)(p)<\nabla_\alpha \nabla_\beta\eta((p(t)),\nu_p(t)>\\
			=&-2g^{\alpha \gamma}(t)(p)g^{\delta \beta}(t)(p)<\partial_\delta\eta(p(t)),\partial_\gamma F_t(p)>h_{\alpha\beta}(t)(p)-<\Delta_{M_t}\eta(p(t)),\nu_p(t)>.
		\end{align*}Since on the sphere of radius $R$, $h_{\alpha \beta}=\frac{g_{\alpha \beta}}{R}$, and using \eqref{tangentialdivergence}, we have
		\begin{align}
			\label{l0}
			H'(0)=-\frac{2}{R}\div_{M}\eta-<\Delta_M\eta,\nu>.
		\end{align}Since $<\Delta_M\eta,\nu>$ does not depend on the choice of coordinates, in the following we choose normal coordinates to do the computation. Let $\eta=\eta^T+\zeta\nu$. Then using normal coordinates we have
		\begin{align}
			\label{l1}
			<\Delta_M\eta^T,\nu>=&\partial_\alpha<\partial_\alpha\eta^T,\nu>-<\partial_\alpha\eta^T,\partial_\alpha\nu>\nonumber\\
			=&-\partial_\alpha<\eta^T,\partial_\alpha\nu>-h_\alpha^\beta<\partial_\alpha\eta^T,\partial_\beta F>\nonumber\\
			=&-\partial_\alpha\left(h_\alpha^\beta<\eta^T,\partial_\beta F>\right)-h_\alpha^\beta<\partial_\alpha\eta^T,\partial_\beta F>\nonumber\\
			=&\frac{1}{R}\left(-\partial_\alpha\left(g_\alpha^\beta<\eta^T,\partial_\beta F>\right)-g_\alpha^\beta<\partial_\alpha\eta^T,\partial_\beta F>\right),\quad \mbox{since $\frac{g_{\alpha\beta}}{R}=h_{\alpha\beta}$}\nonumber\\
			=&\frac{1}{R}\left(-\partial_\alpha<\eta^T,\partial_\alpha F>-<\partial_\alpha\eta^T,\partial_\alpha F>\right),\quad\mbox{since $g_\beta^\gamma=\delta_{\beta \gamma}$ and $\partial_\alpha g_\beta^{\gamma}=0$}\nonumber\\
			=& \frac{1}{R}\left(-2\div_{M}\eta^T-<\eta^T,\Delta_MF>\right)\nonumber\\
			=&-\frac{2}{R}\div_M\eta^T,\quad \mbox{since $\Delta_MF=-H\nu \perp \eta^T$.}
		\end{align}Direct computation also leads to
		\begin{align}
			\label{l2}
			<\Delta_M(\zeta\nu),\nu>=\Delta_M\zeta-\frac{n-1}{R^2}\zeta.
		\end{align}
		Since on the sphere of radius $R$,
		\begin{align*}
			\div_{M}\eta=\div_{M}\eta^T+\frac{n-1}{R}\zeta,
		\end{align*}thus by \eqref{l0}-\eqref{l2}, we obtain \eqref{evoonsphere}.
		
	\end{proof}


\begin{thebibliography}{10}
\bibitem{ACNT}Alvino A, Chiacchio F, Nitsch C, Trombetti C. Sharp estimates for solutions to elliptic problems with mixed boundary conditions. Journal de Mathématiques Pures et Appliquées. 2021 Aug 1;152:251-61.
			
\bibitem{AL25}Acampora P, Lamboley J. Sharp quantitative Talenti's inequality in particular cases. arXiv preprint arXiv:2503.07337. 2025 Mar 10.
				
\bibitem{AN} Alvino A, Nitsch C, Trombetti C. A Talenti comparison result for solutions to elliptic problems with Robin boundary conditions. Communications on Pure and Applied Mathematics. 2023 Mar;76(3):585-603.
\bibitem{AP} Alvino A, Lions PL, Trombetti G. A remark on comparison results via symmetrization. Proceedings of the Royal Society of Edinburgh Section A: Mathematics. 1986 Jan;102(1-2):37-48.				
\bibitem{BW} Bandle C, Wagner A. Second variation of domain functionals and applications to problems with Robin boundary conditions. arXiv preprint arXiv:1403.2220. 2014 Mar 10.
		



\bibitem{Buttazzo20} van den Berg M, Buttazzo G, Pratelli A. On relations between principal eigenvalue and torsional rigidity. Communications in Contemporary Mathematics. 2021 Dec 17;23(08):2050093.

\bibitem{BBN}Bucur D, Buttazzo G, Nitsch C. Symmetry breaking for a problem in optimal insulation. Journal de Mathématiques Pures et Appliquées. 2017 Apr 1;107(4):451-63.

\bibitem{BBN1} 
Bucur D, Buttazzo G, Nitsch C. Two optimization problems in thermal insulation. Notices of the AMS. 2017 Sep;64(8).
		
	\bibitem{Bucur} Bucur D, Giacomini A. Faber–Krahn inequalities for the Robin-Laplacian: A free discontinuity approach. Archive for Rational Mechanics and Analysis. 2015 Nov;218(2):757-824.
	\bibitem{BS25}Barbato L, Salerno F. Talenti comparison results for solutions to $ p $-Laplace equation on multiply connected domains. arXiv preprint arXiv:2504.06103. 2025 Apr 8.
    
\bibitem{Buttazzo} 
Buttazzo G. Thin insulating layers: the optimization point of view. In Proceedings of “Material Instabilities in Continuum Mechanics and Related Mathematical Problems”, Edinburgh 1985 (Vol. 1986, pp. 11-19).
	
	\bibitem{CLW23}Chen D, Li H, Wei Y. Comparison results for solutions of Poisson equations with Robin boundary on complete Riemannian manifolds. International Journal of Mathematics. 2023 Jul 23;34(08):2350045.
		
		\bibitem{CNT24}Celentano A, Nitsch C, Trombetti C. A Talenti comparison result for a class of Neumann boundary value problems. arXiv preprint arXiv:2405.05392. 2024 May 8.
        
		\bibitem{Dam} Dambrine M. On variations of the shape Hessian and sufficient conditions for the stability of critical shapes. Racsam. 2002;96:95-121.
		
		
		\bibitem{DP00} Dambrine M, Pierre M. About stability of equilibrium shapes. ESAIM: Mathematical Modelling and Numerical Analysis. 2000 Jul;34(4):811-34.
		
\bibitem{DNST} Della Pietra F, Nitsch C, Scala R, Trombetti C. An optimization problem in thermal insulation with Robin boundary conditions. Communications in Partial Differential Equations. 2021 Dec 2;46(12):2288-304.	
		
		
		\bibitem{DLW} 
		Du H, Li Q, Wang C. Compactness of $M$-uniform domains and optimal thermal insulation problems. Advances in Calculus of Variations. 2023 Jan 1;16(1):17-43.
		
		\bibitem{H}  Hadamard J. M\'emoire sur le probl\'eme d’analyse relatif l’\'equilibre des plaques \' elastiques encastr\' ees. Imprimerie nationale; 1908.
        \bibitem{He2025} He J, Li Q, Yang H, Wei J. Flow approach on Riesz type nonlocal energies. arXiv preprint arXiv:2505.19655. 
		
        \bibitem{Henrot2}Henrot A. Extremum problems for eigenvalues of elliptic operators. Springer Science \& Business Media; 2006 Aug 29.
		
		
		
		\bibitem{nshap}
	Henrot A. Shape optimization and spectral theory. De Gruyter; 2017.
		
		
		\bibitem{Henrot}
		 Henrot A, Pierre M. Variation et optimisation de formes, Math\'{e}matiques \& Applications (Berlin), 48, Springer, Berlin, 2005.
		
		\bibitem{HLL1} Huang Y, Li Q, Li Q. Stability analysis on two thermal insulation problems. Journal de Mathématiques Pures et Appliquées. 2022 Dec 1;168:168-91.

        \bibitem{H2025}Huang Y, Li Q, Xie S, Yang H. Flow approach on the monotonicity of shape functionals. arXiv preprint arXiv:2502.09485. 2025 Feb 13.
		
		
		\bibitem{Serrin}Serrin J. A symmetry problem in potential theory. Archive for Rational Mechanics and Analysis. 1971 Jan;43:304-18.
		
		\bibitem{Talenti} Talenti G. Elliptic equations and rearrangements. Annali della Scuola Normale Superiore di Pisa-Classe di Scienze. 1976;3(4):697-718.
		
		\bibitem{Hardy} Hardy GH, Littlewood JE, Pólya G. Inequalities. Cambridge university press; 1952.
		
		\bibitem{PS} P\'olya G, Szeg\"o G. Isoperimetric Inequalities in Mathematical Physics.(AM-27).
		
		\bibitem{Lions} Lions PL.  Quelques remarques sur la sym\'{e}trisation de Schwartz,  In Nonlinear Partial Differential Equations and their Application, Coll6ge de France, Seminar 1978 (No. 1, pp. 308-319).

		
		\bibitem{Kesavan} Kesavan S. Some remarks on a result of Talenti. Annali della Scuola Normale Superiore di Pisa-Classe di Scienze. 1988;15(3):453-65.
		
    		
		

	\end{thebibliography}
\end{document}